\documentclass[a4paper]{amsart}

\usepackage{url}
\usepackage{amscd}
\usepackage{amsmath,amsthm,amsxtra,amssymb}
\usepackage{graphics}
\usepackage[dvips]{graphicx}
\usepackage{latexsym}
\usepackage{amsfonts}

\def\R{\mathbb{R}}
\def\C{\mathbb{C}}
\def\Z{\mathbb{Z}}
\def\Q{\mathbb{Q}}

\def\BSO{\mathrm{BSO}}
\def\BSPIN{\mathrm{BSpin}}
\def\BO{\mathrm{BO}}
\def\BPIN{\mathrm{BPin}}

\def\pn{\mathfrak{p}}

\def\vep{\varepsilon}

\def\vep{\varepsilon}

\def\span{\mathrm{span}}
\newcommand{\rank}{\mathop{\mathrm{rank}}\nolimits}

\newcommand{\bbbone}{{\leavevmode\hbox{\small 1\kern-3.3pt\normalsize 1}}}
\newcommand{\bbbzero}{{\leavevmode\hbox{\scriptsize $|$
\kern-3.3pt \normalsize 0\kern-3.3pt
\small $|$}}}

\def\spmapright#1{\smash{%
 \mathop{\hbox to 1.3cm{\rightarrowfill}}
  \limits^{#1}}}
\newcommand{\mapup}[1]{%
  \smash{\mathop{%
   \hbox to 1.5cm{\rightarrowfill}}\limits^{#1}}}
\newcommand{\mapdown}[1]{%
  \smash{\mathop{%
   \hbox to 1.5cm{\rightarrowfill}}\limits_{#1}}}

\theoremstyle{plain}
\newtheorem{theorem}{\textbf{Theorem}}[section]
\newtheorem{corollary}[theorem]{\textbf{Corollary}}

\newtheorem{lemma}[theorem]{\textbf{Lemma}}

\theoremstyle{remark}
\newtheorem{definition}[theorem]{{\rm Definition}}
\newtheorem{example}[theorem]{{\rm Example}}
\newtheorem{remark}[theorem]{{\rm Remark}}

\numberwithin{equation}{section}

\renewcommand{\theequation}%
    {\thesection.\arabic{equation}}

\begin{document}

\date{\today} 
\title{Obstructions to the existence of fold maps}
\author{Rustam Sadykov}
\thanks{The first author has been supported by the
FY2005 Postdoctoral Fellowship for Foreign Researchers, Japan Society for the Promotion of Science; and a 
Postdoctoral Fellowship of Max Planck Institute, Germany}
\address{Department of Mathematics, University of Toronto, Toronto, Ontario,  M5S 2E4, Canada}
\email{rstsdk@gmail.com}
\author{Osamu Saeki}
\thanks{The second
author has been
supported in part by Grant-in-Aid for Scientific Research
(No.~19340018), Japan Society for the Promotion of Science.}
\address{Faculty of Mathematics, Kyushu University, 
Motooka 744, Nishi-ku, Fukuoka 819-0395, Japan}
\email{saeki@math.kyushu-u.ac.jp}
\author{Kazuhiro Sakuma}
\thanks{The third author has been
supported in part by Grant-in-Aid for Scientific Research
(No.~18540102), Japan Society for the Promotion of Science.}
\address{Department of Mathematics, Kinki University, 
Osaka 577-8502, Japan}
\email{sakuma@math.kindai.ac.jp}

\dedicatory{Dedicated to Professor Yoshifumi Ando
on his sixtieth birthday}



\begin{abstract}
We study smooth maps between smooth manifolds
with only fold points as their singularities,
and clarify the obstructions to the existence
of such a map in a given homotopy class
for certain dimensions. The obstructions
are described in terms of characteristic
classes, which arise as Postnikov invariants,
and can be interpreted as primary and secondary
obstructions to the elimination
of certain singularities. We also discuss
the relationship between the existence
problem of fold maps and that of vector
fields of stabilized tangent bundles.
%
\end{abstract}

\maketitle

\section{Introduction}\label{intro}

In 1970 Mather posed the
following question (see \cite{Mather6.5}):
does any element of the
homotopy group $\pi_n(S^p)$, $n \geq p$, contain a fold
map $S^n \to S^p$?
Here, a \emph{fold map} is a smooth map with only 
fold singularities, which
are, in a sense, the simplest
among all generic singularities. Thus the fold maps 
form a reasonable class of maps which
is very close to that of submersions.
The problem was affirmatively
solved by Eliashberg in \cite{Elias1, Elias2}, who
solved it by establishing the $h$-principle
for fold maps (see \cite{EM4, Gromov}) on the $1$-jet level.

In \cite{SS4} the second and the third authors considered a
similar problem for maps between $4$-manifolds and showed
that for a closed orientable $4$-manifold $M$,
the homotopy class of a map $M \to S^4$ contains a
smooth map with only fold and cusp
singularities if and only if
the first Pontrjagin class $p_1(M)$ vanishes.
According to Eliashberg \cite{Elias1}, 
the homotopy class of a map
$M \to S^4$ contains a fold map if and only if
both $p_1(M)$ and the second Stiefel--Whitney
class $w_2(M)$ vanish.
Note that $w_2(M)$ coincides with the Poincar\'e
dual to the $\Z_2$-homology class represented
by the closure of the set of cusp points
of a given generic map $M \to S^4$:
in other words, $w_2$ is 
the so-called {\it Thom polynomial} for cusp singularities.
In a similar fashion, $p_1$ is the Thom polynomial
for the so-called $\Sigma^{2, 0}$ singularities
(see \S\ref{secondary} of the present paper).
Thus we can conclude that in the case of a generic map 
$M\to S^4$ of a closed \emph{orientable} $4$-manifold $M$, the Thom
polynomials are the unique obstructions to the elimination
of singularities except for the fold points (and cusp
points).
In contrast, we will see that an analogous result
does not hold for closed \emph{non-orientable} $4$-manifolds
(see Corollary~\ref{th:3} (ii)). 
In other words we have obstructions other than Thom polynomials. 

We are interested in the following general problem
(see \cite{AVGL}):
given a generic
smooth map $g: M \to N$ between smooth manifolds, under what conditions
does there exist a generic
smooth map homotopic to $g$ which has no singularities
of a prescribed type $\Sigma$?
Obstructions to the elimination of singularities
are, for example, characteristic classes (see \cite{Thom2}), 
homotopy invariants
such as the Hopf invariant (see \cite{SS2}), or
smooth structures of manifolds (see \cite{SS3, SS6}).
In many cases, the $h$-principle holds and
the problem is equivalent to
the existence problem of a corresponding jet section $M \to J^k(M,N)$
covering $g$ and avoiding the singular jets of type
$\Sigma$. Then the primary obstruction to the existence is the
Thom polynomial, which is the Poincar\'e
dual to the homology class represented by the closure
of the singular point set $\Sigma(g)$ of $g$ with type $\Sigma$.
The Thom polynomial does not always
tell us a complete answer to the problem,
since the topological
location of $\Sigma(g)$ in the source manifold $M$ can be nontrivial
even if the Thom polynomial of $\Sigma$ vanishes.
This means that there may be other (co)homological obstructions
to eliminating the singularities of a prescribed type by 
homotopy. 

It is the purpose of the present paper to study the so-called {\it higher order obstructions} which arise as well-defined obstructions in those
cases where the primary obstructions, i.e., Thom polynomials, fail to determine the existence of maps with prescribed singularities.

As far as the authors know, 
there are only a few results about
higher order obstructions to the elimination of singularities.
For example, the first and the second authors
clarified such a secondary obstruction
to eliminating cusp singularities for maps
of closed orientable $4$-manifolds into $3$-manifolds
(\cite{Sadykov2, Sae4}),
and Sz\H{u}cs \cite{Sz} discussed this problem from
a viewpoint of cobordism of maps.

In the present paper we obtain a series of results on
higher order obstructions. For example, we compute
the complete set of obstructions to the existence of
{\it tame fold maps} of non-orientable $4$-manifolds into $\R^3$, i.e., fold maps whose restriction to the set of singular points is an
immersion with trivial normal bundle. Namely we show (see Theorem~\ref{thm:43non-ori}) that for a closed
connected non-orientable $4$-manifold $M$, there exists a tame fold map
$f : M \to \R^3$ if and only if
$W_3(M) = 0$ in $H^3(M; \Z_{w_1(M)})$
and $w_4(M) = 0$ in $H^4(M; \Z_2)$, where
$\Z_{w_1(M)}$ denotes the orientation local
system of $M$ and $W_3$ denotes the $3$rd
Whitney class for twisted coefficients.

%

The paper is organized as follows.
In \S\ref{Andohp}, we recall
the $h$-principle results due to
Eliashberg and Ando for fold maps
and discuss the relationship between the
existence problem of fold maps
and that of vector fields of stabilized
tangent bundles. In \S\ref{secondary},
we give a theorem (Theorem~\ref{th:2}) about the existence
of nowhere linearly dependent sections for 
vector bundles over $4$-dimensional CW
complexes, and deduce some corollaries
about the existence of fold maps.
We prove Theorem~\ref{th:2} using
results by Dold--Whitney \cite{DW},
and interpret our result in terms of
Postnikov decompositions
and their invariants. 
We also study fold maps between
equidimensional manifolds for
low dimensions, using a similar argument.
Furthermore, we interpret
these results about the existence of fold maps
from the viewpoint of elimination
of singularities.
In \S\ref{even4}, we study fold maps
of higher dimensional manifolds into
$\R^4$, using known results about
vector fields. 
In \S\ref{43}, we study tame fold
maps of non-orientable $4$-manifolds into
$3$-manifolds, using
the Postnikov decomposition argument. 
It turned out that the existence problem of 
fold maps is related to characteristic classes
of pin vector bundles. In Appendix, we compute the characteristic
classes in degree $4$ and establish their 
relationship to Pontrjagin and Stiefel-Whitney
classes. These are used to interpret 
our results in terms of Postnikov decompositions. 
The content of the Appendix might be folklore. Nevertheless, 
we included it in our paper, since we could not find 
the assertions in the literature.

Throughout the paper, manifolds
are smooth of class $C^\infty$.
The symbol $\vep^\ell$ denotes the
trivial $\ell$-plane bundle over an appropriate
space (when $\ell = 1$, $\vep$ is also used
in place of $\vep^1$).

The authors would like to thank Boldizs\'ar Kalm\'ar for indicating his recent results on the existence of fold maps to them; and the referee for comments that lead to an improvement of the paper. 

\section{Ando's $h$-principle theorem for fold maps}
\label{Andohp}

Let $f: M \to N$ be a smooth map between manifolds with
$n=\dim M\ge \dim N=p$.
We denote by $S(f)$ the set of singular points of $f$, i.e.\ 
the set of all points $x \in M$ such that $\rank df_x < p$.
A singular point $x \in S(f)$ of $f$ is of \emph{fold type}
if $f$ can be written in a form
$$(x_1, x_2, \ldots, x_n) \mapsto (x_1, x_2, \ldots, x_{p-1},
\pm x_p^2 \pm x_{p+1}^2 \pm 
\cdots \pm x_n^2)$$ 
for some local coordinates around $x$ and $f(x)$.
We say that $f: M \to N$ is a \emph{fold map} 
if all of its singular points are of
fold type.
Note that for $p=1$ a singular point is of fold type
if and only if it is a nondegenerate critical point,
and hence a fold map into $\R$
is nothing but a Morse function.

Fold maps can be characterized in terms of jets as follows. By definition, the \emph{$1$-jet bundle} 
\[
   J^1(M,N)=\mathop{\mathrm{Hom}(TM,TN)} \longrightarrow M
\]
is a fiber bundle whose fiber over a point $x\in M$ consists of all pairs $(y, h)$ of points $y\in N$ and linear maps $T_xM\to T_yN$ of tangent planes. We note that a smooth map $f: M\to N$ determines a jet $j^1_xf$ at each point $x\in M$ with $y=f(x)$ and $h=df_x$. A point of the $1$-jet space $J^1(M,N)$ is called a \emph{$1$-jet}. Let $\Sigma^r$ denote the submanifold of the $1$-jet space $J^1(M,N)$
consisting of the $1$-jets with corank $r$,
where the \emph{corank} of a jet $j^1_xg$ means the
rank of the kernel of the differential $dg_x$.
Then a smooth map $f: M \to N$ is a fold map if and only if
its $1$-jet exten\-sion $j^1f: M \to J^1(M,N)$ is transverse to 
$\Sigma^1$, $j^1f(M) \cap \Sigma^r = \emptyset$, 
$r \geq 2$, and
$f|_{(j^1f)^{-1}(\Sigma^1)}$ is an immersion
(see \cite[Chap.~III, \S 4]{GG} for more details).
Note that $(j^1f)^{-1}(\Sigma^1) = S(f)$ is a closed regular
submanifold of $M$ of dimension $p-1$.
Singularities of fold type are the simplest, i.e.\ they have
the smallest codimension, among all generic corank one
singularities.

As has been mentioned in \S1,
Eliashberg \cite{Elias1,Elias2} studied 
the existence problem of fold maps and 
obtained the $h$-principle of fold maps, which allows us to replace the existence problem of fold maps by an algebraic topology problem. Namely, a fold map $f:M\to N$ exists if and only if there is a section $s: M\to J^1(M,N)$ transversal to $\Sigma^1$ such that 
\begin{itemize}
\item $s(M) \cap \Sigma^r = \emptyset$ for each 
$r \geq 2$, 
\item in a neighborhood $U$ of each point $x$ of $s^{-1}(\Sigma^1)$ there is a fold map $f_x$ such that $s|U=j^1f_x|U$, and  
\item for each index $i=0,1,..., (n-p+1)/2$, there is a point $x\in s^{-1}(\Sigma^1)$ such that $x$ is a fold singular point of $f_x$ of index $i$.  
\end{itemize}
The Eliashberg $h$-principle implies that 
if $M$ is stably parallelizable, then every map
$g : M \to S^p$ (or $g : M \to \R^p$), 
$p \leq \dim M$, is homotopic to
a fold map, which gives a complete
solution to the original problem of Mather mentioned
in \S\ref{intro}.

In general, the stable parallelizability is not a necessary
condition for the existence of fold maps into
$\R^p$.
According to the Thom--Levine theorem
\cite{Lev2, Thom2},
there exists a fold map $f : M \to \R^2$ 
of a closed connected manifold $M$ with
$\dim{M} \geq 2$ if and only if
the Euler characteristic of $M$ is even.
Thus the existence problem of fold maps into $\R^2$
has been completely solved.
For fold maps into $\R^3$, 
the first and the second authors independently
determined necessary and sufficient
conditions for the existence
of such a fold map on a closed oriented $4$-manifold 
\cite{Sadykov2, Sae4}, and when $\dim{M} \ge 5$
the first and the third authors
recently solved the problem except for a few cases
(see \cite{Sadykov4} when $\dim{M}$ is even,
and \cite{Sak3} when $\dim{M}$ is odd).

In this paper we mainly study the existence problem 
of fold maps of even dimensional manifolds into
$\R^4$, which, by Theorem~\ref{Ando} below, 
is equivalent to the existence problem of fold maps in a given 
homotopy class of maps of an even dimensional manifold into $S^4$.   

In the following, we say that a fold map
$f : M \to N$ is \emph{tame}
if the normal bundle of the immersion
$f|_{S(f)}$ is orientable.
Note that if $\dim{M} - \dim{N}$ is even, then
every fold map is tame (for example, see \cite{Sae1}).
In \cite{Sae1} the second author proved that if there is a tame
fold map $f:M\to N$, then there exists a fiberwise epimorphism
$TM\oplus\varepsilon^1\to f^*TN$. On the other hand, using the Eliashberg $h$-principle \cite{Elias1,Elias2} in an essential way, Ando showed that the converse holds true as well. 

\begin{theorem}[Ando \cite{Ando6}]\label{Ando}
Let $g: M \to N$ be a continuous map
between smooth manifolds with $n = \dim M \geq \dim N = p$.
Then there exists a tame fold map $f: M \to N$ homotopic
to $g$ if and only if
there exists a fiberwise epimorphism $TM \oplus \vep^1 \to g^*TN$.
\end{theorem}

This suggests a close relationship between the existence problem of
fold maps and that of vector fields.
In order to clarify the relationship,
let us recall the following definition.

\begin{definition}
Let $\xi$ be a vector bundle over a CW complex.
The maximum number of nowhere linearly dependent sections 
of $\xi$ is called the 
\emph{span} of $\xi$ and is denoted
by $\span(\xi)$. If $\xi$ is the tangent bundle $TM$ of a manifold 
$M$, then the span of $\xi$ is also called the \emph{span} of $M$
and is denoted by $\span(M)$. 
The \emph{stable span} of a vector bundle $\xi$,
denoted by $\span^0(\xi)$,
is the limit of the non-negative non-decreasing
sequence $\{s_n\}$, 
where $s_n+n$ is the span of $\xi \oplus \varepsilon^n$ 
for each $n \ge 0$. Similarly, if $\xi$ is the tangent bundle $TM$
of a manifold $M$, then the \emph{stable span} of $M$ can be 
defined as the number $s$ such that $s+1$ is the span of 
$\xi \oplus \varepsilon$, and is denoted by $\span^0(M)$
(see \cite{KZ, Kos} for more details).
\end{definition}

Let $\nu_N$ be the normal bundle of an embedding of $N$ into $\mathbb{R}^{m}$ for some sufficiently big positive integer $m$. 

\begin{lemma} A fiberwise epimorphism $TM\oplus \varepsilon^1 \to g^*TN$ exists if and only if $\span^0(TM\oplus g^*\nu_N) \ge p + \mathop{\mathrm{dim}}\nu_N -1$. 
\end{lemma}
\begin{proof} A fiberwise epimorphism $TM\oplus \varepsilon^1 \to g^*TN$
extends to a fiberwise epimorphism 
\[
     TM\oplus \varepsilon^1 \oplus g^*\nu_N \longrightarrow g^*TN\oplus g^*\nu_N \cong\vep^m
\]
over $M$, which implies $TM \oplus \varepsilon^1 \oplus
g^* \nu_N \cong \eta \oplus \varepsilon^m$, where $\eta$ is the kernel bundle of the above fiberwise epimorphism. Then, the desired inequality follows, since $m=p+\mathop{\mathrm{dim}}\nu_N$. 
On the other hand, if 
\[
   \span^0(TM\oplus g^*\nu_N) \ge p + \mathop{\mathrm{dim}}\nu_N -1, 
\]
then there is a fiberwise epimorphism 
\[
   TM\oplus \varepsilon^u \longrightarrow g^*TN\oplus \varepsilon^{u-1}
\]
for some positive integer $u$. By dimensional reasoning, such a fiberwise epimorphism is homotopic to the direct sum of a desired fiberwise epimorphism and the fiberwise identity map $\varepsilon^{u-1}\to \varepsilon^{u-1}$. 
\end{proof}


In particular, we have the following:

\begin{corollary}\label{span}
Let $M$ be a smooth $n$-dimensional manifold.
For $p$ with $n \geq p \geq 1$, if
$\span^0(M)\ge p-1$,
then there exists a tame fold map $f: M \to \R^p$.
When $n-p$ is even, $\span^0(M) \ge p-1$ if and only if
there exists a fold map $f: M \to \R^p$.
\end{corollary}

Moreover, we have a homotopy theoretical
interpretation of the stable span as follows.
For an $(n+\ell)$-plane bundle $\xi$
with $\ell \geq 0$ over a CW complex $W$
of dimension $n$, let $c_\xi : W
\to \mathrm{BO}_{n+\ell+1}$ be the classifying
map of $\xi \oplus \varepsilon$.
Consider the natural
fibration (see \cite{Whitehead}):
$$
  V_{p+\ell, n+\ell+1} \longrightarrow \mathrm{BO}_{n-p+1}
  \stackrel{\pi}{\longrightarrow} \mathrm{BO}_{n+\ell+1},
$$
where $V_{m, k}$ is the Stiefel manifold consisting
of all orthonormal $m$-frames in $\R^k$.
Then $\span^0(\xi) \geq p+\ell-1$ if and only if
there exists a continuous
map $\tilde{c}: W \to \mathrm{BO}_{n-p+1}$ such that
the diagram
$$
\setlength{\unitlength}{1mm}
\begin{picture}(80,30)
\put(2,0){\makebox(10,10)[c]{$W$}}
\put(50,22){\makebox(10,10)[c]{$\mathrm{BO}_{n-p+1}$}}
\put(50,0){\makebox(10,10){$\mathrm{BO}_{n+\ell+1}$}}
\put(15,5){\vector(1,0){28}}
\put(53,23){\vector(0,-1){14}}
\put(15,9){\vector(2,1){30}}
\put(55,15){\footnotesize $\pi$}
\put(30,21){\footnotesize $\tilde{c}$}
\put(30,7){\footnotesize $c_\xi$}
\end{picture}
$$
is commutative.

We can study the existence of such a lift 
$\tilde{c}$ of $c_\xi$ by using the
Postnikov decomposition argument. 

In general, given a continuous map $f: E\to B$ of path connected topological spaces with homotopy fiber $F$, the \emph{Postnikov decomposition} or \emph{Moore-Postnikov tower} of $f$ is the decomposition of $f$ of the form 
\[
    E \longrightarrow \cdots \longrightarrow E_2 \longrightarrow E_1 \longrightarrow E_0=B, 
\]
where the composition $E\to B$ is homotopic to $f$; each  map $E_{n+1}\to E_{n}$ is a fibration with fiber $K(\pi_n(F), n)$; each composition $E_n\to B$ induces an isomorphism of homotopy groups in dimensions $>n$, and an injection in dimension $n$; and each composition $E\to E_n$ induces an isomorphism of homotopy 
groups in dimensions $<n$ and a surjection in dimension $n$. 

If $f: E\to B$ is a fibration of CW complexes and $\pi_1(B)$ acts trivially on the homology groups of the fiber, then $f$ admits a Postnikov decomposition where each fibration $E_{n+1}\to E_n$ is principal. 

Now, to lift a continuous map $X\to B$ to a map $X\to E$, we may split the problem into those of constructing consecutive lifts of $X\to B$ to $X\to E_1$, $X\to E_2$ and so on. If $X$ is a finite dimensional CW complex, then the lifts $X\to E$ are in bijective correspondence with the lifts $X\to E_n$ for sufficiently big positive integer $n$. On the other hand, the only obstruction to the existence of a lift of $X\to E_{n-1}$ to $X\to E_n$ is an algebraic invariant, called the $n$-th Postnikov invariant. For details we refer the reader to \cite{T0} and the sections ``Postnikov towers" and ``Obstruction theory" in \cite{Hatcher}.


\begin{remark}\label{rem:s1}
As remarked in \cite{Kos}, we have the
following formula:
$$\mathrm{span}(M \times S^1) =
\mathrm{span}^0(M \times S^1) =
\mathrm{span}^0(M) + 1.$$
This formula is useful for applying known
results about the span of manifolds to
the study of the stable span of manifolds.
\end{remark}

\begin{remark}
According to the theory of wrinkled maps
of Eliashberg--Mishachev \cite{EM1},
a closed connected manifold $M$ with $\dim{M} \geq p$ admits
a wrinkled map into $\R^p$ if and only
if $\span(M) \geq p$. Therefore,
by Corollary~\ref{span}, if $M$ admits a
wrinkled map into $\R^p$, then it
admits a tame fold map into $\R^{p+1}$.
We do not know if we can prove this fact
by a direct geometric construction.
\end{remark}

\begin{remark}
About the
existence of fold maps, some results
have already been obtained, for example,
in \cite{Kalmar, OSS, Sae1}.
\end{remark}

\section{Fold maps between equidimensional manifolds}\label{secondary}

Let us first consider the existence problem of
fold maps between $4$-manifolds.
For this purpose, in view of
the discussion in the previous section, we have only to
find obstructions to the existence of a $(k-1)$-frame\footnote{An
\emph{$\ell$-frame} of a vector bundle means a set of
$\ell$ nowhere linearly dependent sections.}
for an arbitrary
$k$-plane bundle over a CW complex of dimension
$4$ with $k \geq 5$.

In the following, 
for a CW complex $W$, we will denote its $i$-skeleton by
$W^{[i]}$.
Let $\xi$ be a possibly non-orientable
$k$-plane bundle over a CW complex $W$.
A \emph{pin structure}\footnote{This 
corresponds to a $\mathrm{Pin}^+$-structure in
the literature
(for example, see \cite{KT}).}
on $\xi$ is the homotopy class of a $(k-1)$-frame of
$\xi$ over $W^{[1]}$ admitting an extension 
over $W^{[2]}$. When a specific pin structure,
say $\pn$, is given, we call $\xi$
a \emph{pin vector bundle} and denote it by $(\xi, \pn)$.
Let $\BPIN$ denote the 
direct limit of the classifying spaces $\BPIN_k$ of
pin $k$-plane bundles. Note that the cohomology
classes of $\BPIN$ are considered to be characteristic
classes of pin vector bundles.
It is known that there is a class 
$z \in H^4(\BPIN; \Z)$ satisfying
\begin{equation}
2z=p_1 \quad \mbox{\rm and} \quad z \equiv w_4 \pmod{2},
\label{eq:zw}
\end{equation}
where $p_1$ is the first Pontrjagin class and
$w_4$ is the 4th Stiefel--Whitney class (see Appendix).
Since 
$H^3(\BPIN; \Z_2) \cong \Z_2$ (e.g., see the calculation of 
$H^*(\BSPIN; \Z_2)$ in \cite{St}) 
contains only one nontrivial element $w_1^3$ and 
$\mathrm{Sq}^1(w_1^3) \ne 0$, there is only one class $z$ that 
satisfies condition (\ref{eq:zw}).

\begin{theorem}\label{th:2} 
Let $\xi$ be an arbitrary $k$-plane bundle over a
CW complex $W$ of dimension $4$ with $k \geq 5$. 
The primary obstruction to the
existence of a $(k-1)$-frame on $\xi$ is
$w_2(\xi)$, i.e.\ $w_2(\xi)$ vanishes if and only if
$\xi$ admits a pin structure, say $\pn$.
The secondary obstruction is a class 
$z(\xi, \pn) \in H^4(W; \Z)$ for which $2z(\xi,
\pn) = p_1(\xi)$ and 
$z(\xi, \pn) \equiv w_4(\xi) \pmod{2}$. 
Furthermore, $z(\xi) = z(\xi, \pn)$ does not depend
on a particular pin structure $\pn$ on $\xi$.
In other words, 
if $w_2(\xi)=0$, then $\xi$ admits a $(k-1)$-frame
if and only if $z(\xi) = 0$.
\end{theorem}

\begin{proof}
There exists a vector bundle $\bar{\xi}$
of sufficiently high rank
over $W$ such that $\xi \oplus \bar{\xi}$ is
trivial. Let $\eta$ be the line bundle
over $W$ with $w_1(\eta) = w_1(\xi)$. Then
$\xi$ admits a $(k-1)$-frame
if and only if $\xi$ is isomorphic to
$\eta \oplus \varepsilon^{k-1}$. 
Since $k \geq 5$,
this last condition is equivalent to
the triviality of $\eta \oplus \bar{\xi}$.

According to \cite{DW} (see also \cite{EM}), the primary 
obstruction is given by
$w_2(\eta \oplus \bar{\xi}) = w_2(\xi)$.
If this vanishes, then $\eta \oplus \bar{\xi}$
is trivial over $W^{[3]}$.
The secondary obstruction
is given by a class $d$
in $H^4(W; \Z)$, which is the cohomology
class of a difference cocycle, and this class does not depend
on a particular trivialization of $\eta \oplus \bar{\xi}$
over $W^{[3]}$. Furthermore, by \cite[Theorem 2(c)]{DW}, the class $z=-d$ or $z=d$ (depending on the sign convention) satisfies the condition (3.1), and therefore
$d$ can be
identified with $z(\xi)$ up to sign. This completes the
proof.
\end{proof}

\begin{remark}\label{rem:Postnikov}
Let us interpret the above result from the viewpoint of Postnikov decompositions. 
We may assume that $k = 5$.
Let us consider the fibration
\[
    V_{4,5}\longrightarrow \BO_1 \stackrel{\pi} \longrightarrow
    \BO_5.
\]
Since the fundamental group of $\BO_5$ acts 
trivially on the homotopy groups of the fiber
$V_{4,5}$ (see \cite[p.~2]{James}) and $\pi_1(V_{4,5}) \cong \Z_2$ is
abelian, there is a classical Postnikov decomposition of $\pi$. Its
first Postnikov invariant, $w_2$, determines a principal fibration
$r : \BPIN_5 \to \BO_5$. 
We choose a fibration $q$ that covers the map $\pi$
with respect to $r$ so that the homotopy class of 
the restriction $v: V_{4,5} \to K(\Z_2, 1)$ 
of $q$ to a fiber of $\pi$ coincides
with the fundamental class of $V_{4,5}$:
\[
\begin{CD} V_{4,5} @>>> \BO_1  \\
            @VvVV    @VqVV  @. \\
            K(\Z_2, 1) @>>> \BPIN_5 @.\\
            @. @VrVV \\
            @. \BO_5 @>w_2>> K(\Z_2, 2).
\end{CD}
\]
Note that $\pi_2(V_{4,5})= 0$ and $\pi_3(V_{4,5})\cong \Z$.
Hence, for the fiber $F$ of $q$, we have $\pi_1(F)=\pi_2(F)=0$ and
$\pi_3(F)\cong\Z$. Consequently, the next nontrivial Postnikov
invariant is a class in $H^4(\BPIN_5; \Z)
= H^4(\BPIN; \Z)$.
By using spectral sequence arguments,
we can identify this class with $z$ up to sign
(see Appendix).
Therefore, $z(\xi)$ can be considered
as the secondary obstruction in the
sense of the Postnikov decomposition as well.
\end{remark}

Let us apply the above results to
the existence problem of fold maps between
$4$-manifolds.
By Theorem~\ref{Ando}, we 
have the following.

\begin{theorem}\label{th:1} 
Let $g : M \to N$ be a continuous
map between smooth $4$-manifolds,
where $M$ is closed and connected.
Then $g$ is homotopic to a fold map
if and only if
$w_2(TM - g^*TN)\in H^2(M; \Z_2)$ vanishes and
one of the following conditions holds,
where $TM - g^*TN$ denotes the formal difference bundle.
\begin{itemize}
\item[\textup{(i)}]
When $M$ is orientable, 
$p_1(TM-g^*TN) \in H^4(M; \Z)$ vanishes.
\item[\textup{(ii)}] When $M$ is non-orientable,
$w_4(TM - g^*TN) \in H^4(M; \Z_2)$
vanishes.
\end{itemize}
\end{theorem}

\begin{proof}
If $M$ is orientable, then $H^4(M; \Z)
\cong \Z$. Therefore, the vanishing of $z$ is equivalent
to the vanishing of $p_1$.
When $M$ is non-orientable,
the reduction modulo
two $\Z \to \Z_2$ gives rise to an isomorphism 
$H^4(M; \Z)\to H^4(M; \Z_2)$. Then the 
vanishing of $z$ is equivalent
to the vanishing of $w_4$.
\end{proof}

As an immediate corollary, we have the following.

\begin{corollary}\label{th:3} 
A closed connected $4$-manifold
$M$ admits a fold map into $\R^4$ 
if and only if
$w_2(M)\in H^2(M; \Z_2)$ vanishes and
\begin{itemize}
\item[\textup{(i)}]
$p_1(M) = 0$ in $H^4(M; \Z)$ when $M$ is
orientable, or
\item[\textup{(ii)}]
$w_4(M) = 0$ in $H^4(M; \Z_2)$ when
$M$ is non-orientable.
\end{itemize}
\end{corollary}

\begin{example}
Let $\Sigma$ be a closed connected orientable
surface and $F$ a closed connected non-orientable
surface. Then, $F \times \Sigma$ admits
a fold map into $\R^4$ if $F$ has even Euler characteristic,
since then $w_2$ and $w_4$ both vanish.
In fact, it is easy to construct a fold
map $g : F \to \R^2$ if $F$ has even Euler
characteristic. Then the composition
$$F \times \Sigma \spmapright{g \times \mathrm{id}_\Sigma}
\R^2 \times \Sigma \hookrightarrow \R^4$$
gives an explicit fold map, where the last map
is an arbitrary embedding.

On the other hand, if $F$ has odd Euler characteristic,
then $F \times \Sigma$ does not admit a fold map
into $\R^4$, since then $w_2$ does not vanish,
although its $w_4$ vanishes.
The real $4$-dimensional projective space $\R P^4$
does not admit a fold map into $\R^4$,
since $w_4$ does not vanish, although its $w_2$ vanishes.

Let $M$ be the underlying smooth $4$-manifold
of a K3 surface. Then it does not admit a fold
map into $\R^4$, since its $p_1$ does not vanish,
although its $w_2$ vanishes. On the other hand,
$\C P^2 \sharp \overline{\C P^2}$ does not
admit a fold map into $\R^4$, since
its $w_2$ does not vanish, although its $p_1$ vanishes.
\end{example}

%
We can apply the proof of 
Theorem~\ref{th:2} to
the existence problem of fold maps between
higher dimensional manifolds as well, as follows.

\begin{theorem}\label{thm:equi}
Let $g : M \to N$ be a continuous
map between smooth $n$-dimensional manifolds,
where $M$ is closed, connected and $n < 8$.
Then $g$ is homotopic to a fold map
if and only if both
$w_2(TM - g^*TN)\in H^2(M; \Z_2)$ 
and $z(TM - g^*TN, \pn) \in H^4(M; \Z)$
vanish, where $TM - g^*TN$ denotes the formal 
difference bundle,
$\pn$ is a pin structure, and
$z(TM - g^*TN, \pn)$ does not
depend on a particular choice of $\pn$.
\end{theorem}

Theorem~\ref{thm:equi} follows from the final remark
of \cite[\S2]{DW} in view of the interpretation of $z$ observed in the proof of Theorem~\ref{th:2}. In terms of Postnikov decompositions the theorem can be proved as follows.  

\begin{proof}




Recall that the homotopy groups of the space $V_{m-1,m}
\cong \mathrm{SO}_m$ get stabilized as $m$ tends to infinity, since the
inclusion $\mathrm{SO}_m \longrightarrow 
\mathrm{SO}_{m+1}$ induces an isomorphism of
homotopy groups for dimensions less than $m-1$.
Therefore, the comparison of the two fibrations
\[
\begin{CD}
V_{m-1,m} @>>> \BO_1 @>>> \BO_m \\
@VVV @VVV @VVV \\
V_{m,m+1} @>>> \BO_1 @>>> \BO_{m+1}
\end{CD}
\]
shows that the Postnikov invariants and their indeterminacies of
the lower fibration in dimensions $\le m-1$ coincide with the
corresponding Postnikov invariants and indeterminacies of the
upper fibration.  On the other hand, we have
$\pi_4(V_{5,6})=0$, and if $m \ge 6$, then 
$\pi_i(V_{m,m+1}) = 0$ for
$i = 4, 5, 6$, which implies that there are no Postnikov 
invariants in dimensions $5$, $6$ and $7$. Consequently, 
for manifolds of dimensions $5$, $6$ and
$7$, the obstruction coincides with that found in
Theorem~\ref{th:2}.
As a consequence,
we obtain Theorem~\ref{thm:equi} in view
of Remark~\ref{rem:Postnikov}.
\end{proof}

Let us end this section by a remark
from the viewpoint of elimination
of singularities.
Let $M$ be a closed connected $4$-manifold and
$g: M \to \R^4$ a stable map (for details, see \cite{GG},
for example).
Note that the set of stable maps is open and
dense in $C^{\infty}(M, \R^4)$ (see \cite{Mather6}).
For a singular point $x \in S(g)$, there exist local coordinates 
$(x_1,x_2,x_3,x_4)$ and $(y_1,y_2,y_3,y_4)$ around 
$x$ and $g(x)$, respectively, such that
$g$ has one of the following normal forms (see \cite{GG}):
\begin{enumerate}
\item
$y_i=x_i\;(i=1,2,3),\;y_4=x_4^2$ ($x$: fold or $A_1$-type);
\item
$y_i=x_i\;(i=1,2,3),\;y_4=x_4^3+x_1x_4$ ($x$: cusp or $A_2$-type);
\item
$y_i=x_i\;(i=1,2,3),\;y_4=x_4^4+x_1x_4^2+x_2x_4$ 
 ($x$: swallowtail or $A_3$-type);
\item
$y_i=x_i\;(i=1,2,3),\;y_4=x_4^5+x_1x_4^3+x_2x_4^2+x_3x_4$ 
 ($x$: butterfly or $A_4$-type);
\item
$y_i=x_i\;(i=1,2),\;y_3=x_3^2+x_1x_4,\;y_4=x_4^2+x_2x_3$ 
 ($x$: hyperbolic umbilic);
\item
$y_i=x_i\;(i=1,2),\;y_3=x_3^2-x_4^2+x_1x_3+x_2x_4,\;
 y_4=x_3x_4+x_2x_3-x_1x_4$ ($x$: elliptic umbilic).
\end{enumerate}
We denote by $A_k(g)$ the set of $A_k$-type singularities of $g$
for $k=1,2,3,4$ and by $\Sigma^{2,0}(g)$ the set of hyperbolic
and elliptic umbilics of $g$.

The Thom polynomials for the above singularities have been
calculated as follows (for example, see \cite[p.~1119]{SS4}):
\begin{center}
\begin{tabular}{ll}
$[S(g)]_2^* = \bar{w}_1 = w_1$, &
$[\overline{A_2(g)}]_2^* = \bar{w}_1^2 + \bar{w}_2 = 
w_2$, 
\\
$[\overline{A_3(g)}]_2^* = \bar{w}_1^3 + \bar{w}_1\bar{w}_2 = 
w_1w_2$, \rule[0.45cm]{0cm}{0cm} &
$[\overline{A_4(g)}]_2^* = \bar{w}_1^4 + \bar{w}_1\bar{w}_3 =
w_1w_3$, \\
$[\overline{\Sigma^{2,0}(g)}]_2^* = \bar{w}_2^2 + 
\bar{w}_1\bar{w}_3 = w_2^2 + w_1w_3$, 
\rule[0.45cm]{0cm}{0cm} &
$[\overline{\Sigma^{2,0}(g)}]^* = \bar{p}_1 = -p_1$,
\end{tabular}
\end{center}
where $\bar{w}_i$ (resp.\ $\bar{p}_i$) denotes the
$i$-th dual Stiefel--Whitney (resp.\ dual Pontrjagin)
class, and
the symbol $[X]_2^*$ denotes the Poincar\'e dual to the
$\Z_2$-homology class of $M$ represented by $X$.
When $M$ is oriented, the normal
bundle to
the $0$-dimensional submanifold $\Sigma^{2,0}(g)$ has
a canonical orientation and
$[\overline{\Sigma^{2,0}(g)}]^*$
denotes the Poincar\'e dual to the $\Z$-homology
class represented by $\Sigma^{2,0}(g)$.
Note that
$w_1w_3$ always vanishes for a closed $4$-manifold.

Let us consider the elimination problem of
singularities of prescribed types. More
precisely, let us consider the elimination
of singularities other than the fold points.
By virtue of the adjacency of singularities, this
is equivalent to the elimination of cusps.

By Corollary~\ref{th:3}, when $M$ is orientable,
the singularities other than the fold points
can be eliminated if and only if their
associated Thom polynomials all vanish. In other
words, the vanishing of the primary obstructions
is sufficient for the elimination.

On the other hand, when $M$ is non-orientable,
even if all the Thom polynomials other than
that for fold points vanish,
the $4$-th Stiefel--Whitney class may not vanish.
In this sense, $w_4$ can be considered as
the secondary obstruction for the elimination
of cusps, or all
those singularities other than the fold points. 
This observation is also justified by Remark~\ref{rem:Postnikov}.

Note that if $g : M \to \R^4$ is 
a stable map with only $A_k$-type singularities,
then $\langle w_4(M), [M]_2\rangle \in \Z_2$
coincides with the Euler characteristic
modulo two of $\overline{A_2(g)}$, where
$[M]_2 \in H_4(M; \Z_2)$ is the $\Z_2$-fundamental
class of $M$ (see \cite{Fukuda1} and \cite[Theorem~4.1]{SS4}).
This means that $w_4(M) = i_!w_2(\overline{A_2(g)})$,
where $i_! : H^2(\overline{A_2(g)}; \Z_2)
\to H^4(M; \Z_2)$ is the Gysin homomorphism
induced by the inclusion $i : \overline{A_2(g)}
\to M$. In this sense, $w_4$ may be considered
as a higher Thom polynomial (compare this with
\cite{Ando3}, for example).

For maps between $n$-dimensional manifolds,
$n = 5, 6, 7$, it is known that
the Thom polynomial for cusp singularities
coincides with $w_2$. Furthermore, the Thom
polynomial for $\Sigma^{2, 0}$ singularities in
integer coefficient is known to be given by
$p_1 + \beta w_3$,
where $\beta$ is the Bockstein homomorphism
associated with the coefficient exact sequence
$$0 \longrightarrow \Z \stackrel{\times 2}{\longrightarrow} \Z
\longrightarrow \Z_2 \longrightarrow 0$$
 (see \cite{FR, Rimanyi, Ronga}).
Note that the modulo two reduction of $\beta w_3$
is equal to $w_1 w_3$, which is the Thom
polynomial for $A_4$ singularities.
Therefore, if the Thom polynomials for
$A_2$, $A_4$ and $\Sigma^{2, 0}$ singularities
all vanish, then $w_2$ and $p_1 = 2z$ vanish.

However, even if the Thom polynomials vanish
for all singularities, the obstruction $z$ may
not vanish. For example, let us consider
the $n$-dimensional manifold $M = \R P^4 \times S^{n-4}$,
$n = 5, 6, 7$. Then, it is easy to see that
all the Thom polynomials for singularities
of codimension up to $4$ vanish.
Singularities of codimension $5$, $6$ or $7$ may
appear and their Thom polynomials are
known (see \cite{FR, Rimanyi}, for example).
However, their Thom polynomials for
$M$ all vanish, since the polynomials
in characteristic classes of $M$
of degree $\geq 5$ all vanish.
On the other hand, $w_4(M)$ does not
vanish, which implies that $z$, whose
modulo two reduction coincides with $w_4$, 
does not vanish. Therefore, for $M$,
all the relevant Thom polynomials vanish,
although it does not admit a fold map into $\R^n$.
This means that the obstruction $z$ is
essentially secondary.

\section{Fold maps of even dimensional manifolds into $\R^4$}
\label{even4}

In the case of maps into $\R^4$ of an even
dimensional manifold $M$, the existence problem of a fold map is
closely related to that of a $3$-frame on
$M$. Let us recall that a fold map of
$M$ into $\R^4$ exists if and only if $\span^0(M) \geq 3$,
i.e., $TM \oplus \varepsilon$ admits
at least four nowhere linearly dependent sections.

To simplify the formulation of statements, in all theorems 
of this section we will assume that the manifold $M$ 
is closed and connected. Let us first recall
the following theorem.

\begin{theorem}[Eagle; Koschorke \cite{Kos}] \label{EK}
If $M$ is an even dimensional manifold
and its Euler characteristic
$\chi(M)$ vanishes, then $\span(M) = 
\span^0(M)$ always holds.
\end{theorem}

First, let us consider the case where the 
dimension of the manifold $M$ is a multiple of $4$.

\begin{theorem}[Atiyah--Dupont \cite{AD}]\label{th:AD} 
An oriented manifold $M$ of dimension $n = 4k$, with $k > 1$, 
admits a $3$-frame if and only if $w_{n-2}(M) = 0$, $\chi(M) = 0$ 
and the signature $\sigma(M)$ of $M$ is divisible by $8$.
\end{theorem}

By using Theorem~\ref{th:AD}, we can prove the
following.

\begin{theorem}\label{th:4k} 
An oriented manifold $M$ of dimension $n = 4k$, with $k > 2$, 
admits a fold map into $\R^4$ if and only if 
$w_{n-2}(M) = 0$ and the signature $\sigma(M)$ of $M$ 
is divisible by $8$.
\end{theorem}

\begin{proof}
In the case where $\chi(M) = 0$ the assertion follows 
from Theorem~\ref{Ando} and the Atiyah--Dupont theorem. 
If $\chi(M)$ is odd, then neither the signature of $M$ is
divisible by $8$, nor the manifold $M$ admits a stable $3$-frame. 
Hence, we may assume that $\chi(M)$ is even.

Let us first consider the case where $\chi(M) > 0$. 
Suppose that the statement of the theorem has
been proved for all manifolds with Euler characteristic $\chi(M)-2$. 
In particular, we assume that the theorem holds true for the 
connected sum $M \sharp S$, where $S$ stands for 
$S^{2k-1} \times S^{2k+1}$.

If $M$ admits a fold map, then so does $M \sharp S$.
Hence, by the inductive assumption $w_{4k-2}(M \sharp S)$ 
vanishes and the
signature of $M \sharp S$ is divisible by $8$. Consequently,
$w_{4k-2}(M) = 0$ and $\sigma(M) \equiv 0 \pmod{8}$, as it has been
claimed.

Conversely, if $w_{4k-2}(M) = 0$ and $\sigma(M) \equiv 0 \pmod{8}$,
then the same equalities hold for $M \sharp S$ as well.
Consequently, by the inductive assumption, the bundle 
$T(M \sharp S) \oplus \varepsilon$ admits a $4$-frame. 
On the other hand, since $\pi_i(V_{4,4k+1}) = 0$ for 
$i \le 4k-4$, every stable $3$-frame
over $S^{2k-1} \vee S^{2k+1}$ is homotopic to the trivial stable
$3$-frame, provided that $k \ge 3$; and therefore a stable 
$3$-frame over $M \sharp S$ gives
rise to at least one stable $3$-frame over $M$.

In the case where $\chi(M) < 0$, we can prove the 
statement by a similar induction, considering 
$M \sharp (S^{2k} \times S^{2k})$ 
whose Euler characteristic is equal to $\chi(M) + 2$.
\end{proof}

\begin{remark} Our technique does not apply to manifolds of dimension $8$, since in the case $k=2$ there exists a non-trivial stable $3$-frame over $S^{2k-1}\vee S^{2k+1}$.
\end{remark}

In the case of a manifold $M$ of dimension $4k+2$, we make use of
a theorem which is due to Atiyah--Dupont \cite{AD} (in the
orientable case) and Koschorke \cite{Kos} (in the non-orientable
case).

\begin{theorem}\label{th:ADK} 
A manifold $M$ of dimension $n = 4k+2$, with $k > 0$ 
in the orientable case and $k > 1$ in the non-orientable case, 
admits a $3$-frame if and only if $\chi(M) = 0$ and $w_{n-2}(M) = 0$.
\end{theorem}

In view of Theorem~\ref{th:ADK}, the argument similar 
to that in the proof of Theorem~\ref{th:4k} yields a condition 
for the existence of a fold map into $\R^4$ of a manifold 
of dimension $4k+2$ as follows.

\begin{theorem} 
A manifold $M$ of dimension $n = 4k+2$, with $k > 1$, admits a fold
map into $\R^4$ if and only if $w_{n-2}(M) = 0$.
\end{theorem}

\begin{proof}
First note that we have
$$Sq^2(w_{4k}) = w_2 w_{4k} + w_{4k+2}$$
by virtue of the Wu formula (for example,
see \cite{MS}).
Therefore, if $w_{4k}$ vanishes, then so does $w_{4k+2}$.
Then the rest of the argument is similar to that
in the proof of Theorem~\ref{th:4k}.
\end{proof}

\begin{remark} Our technique does not apply to manifolds of dimension $6$, since in the case $k=1$ there exists a non-trivial stable $3$-frame over $S^{2k+1}\vee S^{2k+1}$.
\end{remark}

\section{Fold maps of non-orientable
$4$-manifolds into $\R^3$}\label{43}

As has already been mentioned in \S\ref{Andohp},
the existence problem of a fold map into $\R^3$
for closed orientable $4$-manifolds
has been solved by the first and the second authors,
independently.
In this section, let us consider the existence
problem of fold maps into $\R^3$ for closed
non-orientable $4$-manifolds.

It immediately follows from Theorem~\ref{Ando} and
Corollary~\ref{th:3} that a closed
non-orientable $4$-manifold $M$ admits
a tame fold map $f: M \to \R^3$ if
$w_2(M) = 0$ and $w_4(M) = 0$.
For example,
$$M=\sharp^{2\ell} \R P^4 \sharp^k (S^2 \times S^2)
\sharp^m (S^1 \times S^3)$$
with $\ell \ge 1$, $k \ge 0$ and $m \ge 0$
admits a fold map into $\R^3$.
However, the same argument does not tell us if
$\sharp^{2\ell+1}\R P^4$ admits
a fold map into $\R^3$ or not.

Nevertheless, for the existence of tame
fold maps, we have the following.

\begin{theorem}\label{thm:43non-ori}
Let $M$ be a closed connected non-orientable $4$-manifold.
Then there exists a tame fold map
$f : M \to \R^3$ if and only if
$W_3(M) = 0$ in $H^3(M; \Z_{w_1(M)})$
and $w_4(M) = 0$ in $H^4(M; \Z_2)$, where
$\Z_{w_1(M)}$ denotes the orientation local
system of $M$ and $W_3$ denotes the $3$rd
Whitney class for twisted coefficients.
\end{theorem}

Using Theorem~\ref{thm:43non-ori}, we see that
$\sharp^{2\ell+1}\R P^4$ does not admit a
tame fold map into $\R^3$.

\begin{remark}  An argument similar to that in the proof of Theorem 4.3
cannot be used to deduce Theorem 5.1 from a result
on the span of $4$-manifolds (see Remark 5.4) as there exist
non-trivial $3$-frames in the trivial vector $5$-bundle over $S^1 \vee S^3$.
\end{remark}

\begin{proof}[Proof of Theorem~\textup{\ref{thm:43non-ori}}]
In view of Theorem~\ref{Ando},
we have only to show that the
stable span of $M$ satisfies
$\span^0(M) \geq 2$ if and only if
both $W_3(M)$ and $w_4(M)$ vanish.

Recall that
$\span^0(M) \ge 2$ if and only of
the map $t_M: M \to \mathrm{BO}_5$
classifying $TM \oplus \vep^1$ admits a lift with respect to
the fibration
$$
   V_{3,5} \longrightarrow \mathrm{BO}_2 \longrightarrow\mathrm{BO}_5.
$$
It is known that $\pi_2(V_{3,5}) \cong \pi_3(V_{3,5}) \cong \Z$
and that the action of $\pi_1(BO_5)$ on
$\pi_2(V_{3,5})$ is nontrivial.
Hence, the primary obstruction to the existence of such a lift is
the Whitney class $W_3(M) \in H^3(M; \Z_{w_1(M)})$
(for this and the following arguments,
the reader is referred to \cite{Pollina, Randall}).

Suppose that $W_3(M)=0$.
Then there is a lift over a $3$-skeleton of $M$.
Recall that the action of $\pi_1(BO_5)$ on
$\pi_3(V_{3,5})$ is trivial.
Let $E$ be the space in the second stage of the Postnikov decomposition
and $k \in H^4(E; \Z)=[E, K(\Z,4)]$ be the second Postnikov
invariant:
$$
\begin{CD}  V_{3,5} @>>> \mathrm{BO}_2 @. \\
            @.     @V q VV  @. \\
            K(\Z, 2) @>>> E @>k>> K(\Z,4)\\
            @.  @V r VV \\
            M@>t_M>> \mathrm{BO}_5 @>W_3>> L(\Z, 3).
\end{CD}
$$
Since $M$ is non-orientable, the modulo two reduction
$H^4(M; \Z) \to H^4(M; \Z_2)$ is an
isomorphism. Therefore, we have only to
consider the modulo two reduction
of the Postnikov invariant.

We see that the modulo two reduction 
$\bar{k}$ of $k$ coincides with $r^* w_4$.
Indeed, the modulo two reduction homomorphism
$H^4(E;\Z) \to H^4(E;\Z_2)$
is associated to the obvious homomorphism
$\pi_3(V_{3,5}) \cong \Z \to \pi_3(V_{2,5}) \cong \Z_2$
of coefficient groups.

To calculate the indeterminacy we need to consider the map
$$
  \alpha: K(\Z, 2) \times E \to E
$$
given by the action of the fiber $K = K(\Z,2)$ on the fibration
$r: E \to \mathrm{BO}_5$.
Then we can express the class $\alpha^*\bar{k}$ in the form
$$
  \alpha^*\bar{k} = (1 \times x)+(at \times y)+(bt^2 \times 1)
$$
for some $a, b \in \Z_2$ and for
some $x \in H^4(E; \Z_2)$ and $y \in H^2(E; \Z_2)$,
where $t \in H^2(K; \Z_2) = \Z_2$
is the generator.

We know that the inclusion $K \to E$ in the fibration
$r: E \to \mathrm{BO}_5$ takes $k$ to the 
Postnikov invariant of the fibration
$V_{3,5} \to K$. Let $F$ denote the fiber
of this fibration. Since $F$ is
$2$-connected, we have the
Serre long exact sequence
$$
H^3(K; \Z_2) \to
H^3(V_{3,5}; \Z_2) \to H^3(F; \Z_2) \to H^4(K; \Z_2),
$$
where the map $H^3(F; \Z_2) \to H^4(K; \Z_2)$ 
corresponds to the transgression we
are interested in, and the other maps are the homomorphisms
induced by the inclusion of the fiber, 
and the fiber bundle projection.
It is easy to see that the
above sequence has the form $0 \to \Z_2 \to
\Z_2 \to \Z_2$, which implies that
the modulo two Postnikov invariant of the fibration
$V_{3,5} \to K$ is trivial.
Hence, 
considering the composition
$K \to K \times E \to E$, we conclude that $b=0$.

Next, by considering the composition $E \to K \times E \to E$
we conclude that $x=r^*w_4$.

Finally, let us determine $a$ and $y$. In fact, we will show that the term $at\times y$ is trivial. Since the homotopy fiber of $q$ is $2$-connected, the map $q$ induces an isomorphism of integral cohomology groups in dimension two. Hence $H^2(E; \Z)$ contains only one non-zero element, namely $(q^*)^{-1}(\beta w_1)$, which coincides with $r^*(\beta w_1)$. Since $y$ is the modulo two reduction of some class in $H^2(E; \Z)$, it follows that $y$ is a multiple of $r^*(w_1^2)$. If $y$ is zero, then $at\times y$ is trivial as it has been claimed. 

Suppose that $y$ is not zero, i.e., $y=r^*(w_1^2)$. 

Let us consider the non-orientable $4$-manifold $\R P^4$,
which satisfies $W_3(\R P^4) = 0$ and $w_4(\R P^4) \neq 0$.
In particular, we have $\span^0(\R P^4) = 0$ and hence
the secondary obstruction for $\R P^4$ must be nontrivial.
Let $\tilde{t} : \R P^4 \to E$ be a lift of $t_{\R P^4}$.
For a map $s : \R P^4 \to K$, we have
$$(s, \tilde{t})^*\alpha^*\bar{k} = 
w_4(\R P^4) + a s^*t \smile w_1(\R P^4)^2,$$
where $(s, \tilde{t}) : \R P^4 \to K \times E$
is the map that sends
$p \in \R P^4$ to $(s(p), \tilde{t}(p))$. 
On the other hand, for the map $g: \R P^4 \to K$ 
corresponding to
$\beta w_1(\R P^4) \in H^2(\R P^4; \Z)$,
we have $g^*t=w_1(\R P^4)^2$.
Since it holds that $w_1^4=w_4$ for $\R P^4$ (see, for example, 
\cite{MS}), if $a \neq 0$, then this implies that the secondary
obstruction vanishes, which 
is a contradiction. Thus we have $a = 0$.
Therefore, $\alpha^*\bar{k} = 1 \times r^*w_4$ and
we conclude that $w_4$ is the secondary obstruction.
\end{proof}

\begin{remark}\label{rem:simple43}
We can also prove Theorem~\ref{thm:43non-ori} in
a simpler way as follows. Let us consider
the exact sequence 
$$H^2(M; \Z_{w_1(M)}) \stackrel{\times 2}{\longrightarrow}
H^2(M; \Z_{w_1(M)}) \spmapright{\mathrm{mod}\ 2} 
H^2(M; \Z_2) \stackrel{\tilde{\beta}}{\longrightarrow} 
H^3(M; \Z_{w_1(M)})$$
associated with the coefficient
exact sequence
$$0 \longrightarrow \Z_{w_1(M)} \stackrel{\times 2}
\longrightarrow 
\Z_{w_1(M)}
\longrightarrow \Z_2 \longrightarrow 0.$$
If there exists a tame fold map
$f : M \to \R^3$, then the singular set
$S(f)$ is an orientable surface.
Since the Poincar\'e dual to the $\Z_2$-homology
class represented by $S(f)$ coincides with
$w_2(M)$ (see \cite{Thom2}), this
implies $W_3(M) = \tilde{\beta} w_2(M) = 0$
(see \cite{Cadek}).
Furthermore, it is known that the Euler
characteristic of $M$ has the same
parity as that of $S(f)$ (see \cite{Fukuda1}).
Therefore, both $W_3(M)$ and $w_4(M)$ vanish.

Conversely, suppose that both $W_3(M)$ and $w_4(M)$ vanish.
Since $W_3(M) = \tilde{\beta} w_2(M) = 0$,
there exists a cohomology class $e \in H^2(M; 
\Z_{w_1(M)})$ whose modulo two reduction
coincides with $w_2(M)$. Then we see easily that
there exists a $2$-plane
bundle $\eta$ over $M$ with $w_1(\eta) = w_1(M)$
such that its Euler class $e(\eta)
\in H^2(M; \Z_{w_1(M)})$ coincides with $e$.

In order to show that there exists
a fiberwise epimorphism $TM \oplus \vep^1
\to \vep^3$, it suffices to show that
$TM \oplus \vep^1 \cong \eta \oplus \vep^3$.
For this, we have only to show that
$\xi = \eta \oplus \nu_M$ is trivial, where
$\nu_M$ is the normal bundle of an embedding
of $M$ into a Euclidean space of sufficiently
big dimension.

By our construction, it is easy to see
that both $w_1(\xi)$ and $w_2(\xi)$ vanish.
Therefore, $\xi$ has a trivialization
over the $3$-skeleton of $M$.
It is known (see \cite[p.~2]{James}) that the 
principal $\mathrm{SO}_5$-bundle over $M$ associated to 
$\xi$ is simple. Hence, the obstruction 
to extending the trivialization over the whole of $M$ is an
element in $H^4(M; \pi_3(\mathrm{SO}_5))$, which, in view of the
isomorphism
\[
   H^4(M; \pi_3(\mathrm{SO}_5)) \longrightarrow H^4(M; \pi_3(V_{2,5})),
\]
can be interpreted as the obstruction to the existence of a
$2$-frame for $\xi$. In other words, the obstruction coincides with
$w_4(\xi) = w_4(M)$.

Alternatively, we can use a result of \cite{DW}
to show that $\xi$ is trivial.
\end{remark}

\begin{corollary}
Let $g : M \to \R^3$ be a stable map
of a closed connected $4$-manifold $M$
such that the singular set $S(g)$ is
orientable. Then there exists
a tame fold map of $M$ into $\R^3$.
\end{corollary}

\begin{proof}
By an argument as in Remark~\ref{rem:simple43}, we see
that both $W_3(M)$ and $w_4(M)$ vanish.
Then the result follows from Theorem~\ref{thm:43non-ori}.
\end{proof}

\begin{remark}
By Randall \cite[Corollary~3.7]{Randall} 
(see also \cite{Kos, Pollina}), 
a closed connected non-orientable $4$-manifold $M$
satisfies $\span(M) \geq 2$ if and only if $W_3(M) = 0$ and
its Euler characteristic vanishes. 
\end{remark}

\begin{remark}
In \cite{Sae1}, the second author
constructed a fold map $f : M \to \R^3$
of the total space $M$ of an $\R P^2$-bundle
over $\R P^2$. Note that $w_4(M)$ does
not vanish and that this fold map is not tame.
This means that the vanishing of $w_4$ is not
necessary for the existence of a fold map.
\end{remark}

\begin{remark}
Suppose that a closed non-orientable
$4$-manifold $M$ admits a fold map into $\R^4$.
Then we have $\mathrm{span}^0(M) \geq 3$, and
therefore $M$ admits a tame fold map into
$\R^3$ by Corollary~\ref{span}.
This observation is easily seen to be
consistent with Corollary~\ref{th:3} and
Theorem~\ref{thm:43non-ori}.
\end{remark}

\begin{remark}
We do not know if $W_3$ or $w_4$ can be
interpreted as a certain Thom polynomial
for some singularities.
\end{remark}

As to fold maps of $6$-dimensional manifolds
into $\R^3$, we have the following.

\begin{theorem}\label{thm:63}
Let $M$ be a closed $6$-dimensional manifold.
If $M$ is orientable, then there always
exists a tame fold map $f : M \to \R^3$.
If $M$ is non-orientable and $W_5(M) = 0$
in $H^5(M; \Z_{w_1(M)})$, then
there exists a tame fold map $f : M \to \R^3$.
\end{theorem}

\begin{proof}
When $M$ is orientable,
according to \cite{Randall}, the span of the $7$-dimensional
manifold $M \times S^1$ satisfies
$\mathrm{span}(M \times S^1) \geq 3$.
This implies $\mathrm{span}^0(M) \geq 2$
by Remark~\ref{rem:s1}. Therefore, by
Corollary~\ref{span}, we have the desired conclusion.

When $M$ is non-orientable, $\mathrm{span}(M
\times S^1) \geq 3$ if $W_5(M \times S^1)$
vanishes. Thus we have the desired conclusion.
\end{proof}

\begin{remark}
In \cite{Sak3}, a slightly weaker
version of the above theorem is given.
\end{remark}

\begin{remark}
For $n \ge 5$,
the existence problem of fold maps into
$\R^3$ for closed $n$-dimensional manifolds has been
almost completely solved by the first and the third authors.
In \cite{Sak3}, when $n$ is odd,
it is proved
that a closed $n$-dimensional manifold
$M$ admits a fold map into $\R^3$ if and only if $w_{n-1}(M)=0$.
Note that $w_{n-1}(M)$ coincides with 
the Thom polynomial for cusp singularities.
On the other hand, in \cite{Sadykov4}, it is proved
that a closed orientable $n$-dimensional
manifold always admits
a fold map into $\R^3$ when $n$ is even and $n \ge 8$.
For $n = 6$, see Theorem~\ref{thm:63}.
\end{remark}

\section{Appendix}

In this section we prove claims made in the beginning of \S\ref{secondary} and in Remark~\ref{rem:Postnikov}. These assertions might be folklore, but as we could not find them in the literature, we decided to include them in this paper as Appendix. 

In this section we continue to use notations introduced in Remark~\ref{rem:Postnikov}.

To begin with let us identify the first two Postnikov invariants of the obvious fibration $p$: 
\[
    V_{4,5}\longrightarrow \BSO_1 \stackrel{p}\longrightarrow
    \BSO_5.
\]
The first one is known to be the classical obstruction class
$w_2$. Consequently, the space in the second stage of the
Postnikov decomposition of $p$ is nothing but $\BSPIN_5$. The next
nontrivial Postnikov class is in degree $4$ as
$\pi_2(V_{4,5})=0$, while $\pi_3(V_{4,5})\cong\Z$. Since
$\tilde H^*(\BSO_1; \Z)=0$, we
conclude that the second Postnikov invariant is a generator
$\tilde z$ of $H^4(\BSPIN_5; \Z)\cong\Z$, which is known to coincide up
to sign with half of the first Pontrjagin class of the universal
spin bundle~\cite{Th}. Let us also observe that since $\tilde z$
is not divisible by two, its reduction modulo $2$ is the unique
non-trivial element $w_4$ in $H^4(\BSPIN_5; \Z_2)$.

\begin{lemma} The group $H^4(\BPIN_5; \Z)$ is isomorphic to $\Z\oplus \Z_2$.
It is generated by a class $z$ with $r^*p_1=2z$, and
$y=r^*(\beta w_1)^2$. The transgression of $q$ takes the
fundamental class of $F$ to $\pm z$.
\end{lemma}
\begin{proof} Let us consider the fibration $K(\Z_2,1)\to
\BPIN_5\stackrel{r}\longrightarrow \BO_5$. It is known \cite{Br},
\cite{Fe} that the group $H^4(\BO_5; \Z)\cong\Z\oplus\Z_2\oplus\Z_2$
is generated by $p_1, (\beta w_1)^2$ and $\beta(w_1w_2)$. Since $\tilde{H}^*(K(\Z_2,1); \Q)=0$, the Serre exact sequence implies that $r$ induces an isomorphism of rational cohomology groups
\[
      H^i(\BPIN_5; \Q)\cong H^i(\BO_5; \Q)
\]
for all $i$. In particular, 
the class $x=r^*p_1$ in $H^4(\BPIN_5; \Z)$ is of infinite order. On the other
hand, the map $q^*$ takes $y=r^*(\beta w_1)^2$ to $(\beta
w_1)^2\in H^4(\BO_1; \Z)$. Hence $y\ne 0$. Consequently,
$H^4(\BPIN_5; \Z)$ contains an element of infinite order and an
element of order two.

Let us consider, now, the fibration $F\to \BO_1\stackrel{q}\longrightarrow \BPIN_5$. We have
$H^1(F; \Z)=H^2(F; \Z)=0$, $H^3(F; \Z)\cong\Z$, $H^4(\BO_1; \Z)\cong\Z_2$ and the exact sequence
\begin{equation}\label{eq:6}
0 \longrightarrow H^{3}(F; \Z) \stackrel{\tau}\longrightarrow
H^4(\BPIN_5; \Z) \longrightarrow E_{\infty}^{4,0} \longrightarrow
0,
\end{equation}
where $\tau$ is the transgression. Let us recall that $H^3(F; \Z)\cong\Z$ and the group $H^4(\BPIN_5;
\Z)$ contains an element of infinite order and an element of order
two. Hence, from the exact sequence~(\ref{eq:6}), it follows that the
group $E_{\infty}^{4,0}$ is not zero. On the other hand, 
we observe that the group $E_{\infty}^{4,0}$ is a term of the Serre spectral sequence associated with the fibration $q$. Consequently, by the Serre Theorem (e.g, see \cite[Theorem 1.14(b)]{Hatch}),  the group $E_{\infty}^{4,0}$ is isomorphic to a subgroup
of $H^4(\BO_1; \Z)\cong\Z_2$. Hence $E_{\infty}^{4,0}\cong\Z_2$.
Since $H^4(\BPIN_5; \Z)$ contains an element of order two, the
exact sequence~(\ref{eq:6}) implies now that $H^4(\BPIN_5; \Z)\cong \Z\oplus
\Z_2$ with the first factor generated by the image $z$ of the fundamental class of $F$ by the
transgression. Let $y$ stand for $r^*(\beta w_1)^2$, which is the
unique element of order two in $H^4(\BPIN_5; \Z)$. Then, for some
integer $n$, the class $r^*p_1$ has the form $nz+\alpha y$, where
$\alpha$ equals $0$ or $1$. The classes $r^*p_1$ and $z$ are in
the kernel of the homomorphism $q^*$, while $q^*y$ is the
generator of $H^4(\BO_1; \Z)$. Hence $\alpha=0$ and $r^*p_1$ is a
multiple of $z$.

There is a map of the Postnikov decomposition of $\BSO_1\to \BSO_5$
into the Postnikov decomposition of $\BO_1\to \BO_5$. Such a map
takes the class $z$ to the class $\tilde z$. Hence $r^*p_1=\pm
2z$ (see the argument just before Lemma 6.1). If necessary, we may substitute $z$ by $-z$ so that we have $2z=
r^*p_1$.
\end{proof}

Next we determine the reduction modulo two of the class $z$.

\begin{lemma} The homomorphism of cohomology groups associated with the
homomorphism $\Z\stackrel{\mathrm{mod}\ 2}\longrightarrow \Z_2$
takes $z$ to the class $w_4(r^*\gamma)$, where $\gamma$ is the
universal $5$-plane bundle over $\BO_5$.
\end{lemma}
\begin{proof} The bundle $r^*\gamma$ is the universal pin vector bundle, while
the class $z$ is the obstruction to extending the universal pin
structure of $r^*\gamma$ to a $4$-frame of $r^*\gamma$ over a $4$-skeleton of $\BPIN_5$. Since
$\pi_2(V_{4,5})=0$, there is an extension of the universal
pin structure to a $4$-frame over
a $3$-skeleton of $\BPIN_5$. Furthermore, by \cite[p.\,2]{James}, the action of
$\pi_1(\BPIN_5)$ on $\pi_3(\BPIN_5)$ is trivial. Hence, the class
$z$ is the classical obstruction, $z\in H^4(\BPIN_5;
\pi_3(V_{4,5}))$. Furthermore, the modulo two reduction
homomorphism is associated to the obvious homomorphism
\[
  \pi_3(V_{4,5}) \longrightarrow \pi_3(V_{2,5}).
\]
Hence, it takes the class $z$ to the obstruction to the
existence of a $2$-frame of $r^*\gamma$ over
the $4$-skeleton of $\BPIN_5$, i.e., to the class
$w_4(r^*\gamma)$.
\end{proof}


\end{document}